\documentclass[a4paper,10pt]{article}

\usepackage{amsfonts,amssymb,amsmath}

\numberwithin{equation}{section}
\usepackage[labelsep=endash]{caption}

\usepackage{tikz}
\usetikzlibrary{calc,arrows.meta}

\newtheorem{theo}{Theorem}[section]
\newtheorem{coro}[theo]{Corollary}
\newtheorem{lemm}[theo]{Lemma}
\newtheorem{prop}[theo]{Proposition}
\newtheorem{defi}[theo]{Definition}
\newtheorem{rema}[theo]{Remark}
\newtheorem{exam}[theo]{Example}

\newenvironment{proof}{\noindent \textbf{{Proof.}} \sf}

\def\qed{\hfill $\diamond$ \bigskip}

\def\lim{\mathop{\rm lim}\nolimits}

\def\HH{\mathsf H}
\def\HHH{\mathsf{HH}}

\def\Ext{\mathsf{Ext}}
\def\Hom{\mathsf{Hom}}
\def\Tor{\mathsf{Tor}}

\def\dim{\mathsf{dim}}

\begin{document}
\sf

\title{Deleting or adding arrows of a bound quiver algebra and Hochschild (co)homology}
\author{Claude Cibils,  Marcelo Lanzilotta, Eduardo N. Marcos,\\and Andrea Solotar
\thanks{\footnotesize This work has been supported by the projects  UBACYT 20020130100533BA, PIP-CONICET 11220150100483CO, USP-COFECUB.
The third mentioned author was supported by the thematic project of FAPESP 2014/09310-5 and acknowledges support from the "Brazilian-French Network in Mathematics". The fourth mentioned author is a research member of CONICET (Argentina) and a Senior Associate at ICTP.}}

\date{}
\maketitle
\begin{abstract}
We describe how the Hochschild (co)homology of a bound quiver algebra changes when deleting or adding arrows to the quiver. The main tools are relative Hochschild (co)homology, the Jacobi-Zariski long exact sequence obtained by A. Kaygun and a length one relative projective resolution of tensor algebras.

\end{abstract}



\section{\sf Introduction}

Hochschild (co)homology has been widely studied for algebras presented by quivers and relations, that is for bound quiver algebras. The main purpose of this work is to make precise how the Hochschild (co)homology changes when deleting or adding arrows.

In 1956 G. Hochschild  \cite{HOCHSCHILD1956} introduced a relative (co)homological theory, which amounts to consider an exact category, see \cite{QUILLEN,BUHLER}. This theory has been used rarely, in part because of a lack of relation with the usual (co)homological theory given in \cite{HOCHSCHILD1945}.  Nevertheless, for an extension of algebras $B\subset A$ such that $A/B$ is a projective  $B$-bimodule,  A. Kaygun has established in 2012 \cite{KAYGUN,KAYGUNe}  a Jacobi-Zariski long exact sequence connecting both theories.  Note that M. Auslander and \O. Solberg have considered relative homological algebra, see \cite{AUSLANDERSOLBERG1} and  \cite{SOLBERG}.

In Section \ref{account} we provide an account of the relative (co)homological theory and we show that for a tensor  algebra there is a relative projective resolution of length one.

An inert arrow of a bound quiver algebra $A$ is an arrow of the quiver which is not involved in a minimal set of relations of $A$. Our main result in Section \ref{homology} is that deleting a set of inert arrows does not change the Hochschild homology in degrees greater or equal than $2$.

On the other hand in Section \ref{cohomology} we provide formulas for Hochschild cohomology, also when deleting inert arrows of a bound quiver algebra. For degrees greater or equal than $2$,  the change of dimension is expressed in terms of the dimensions of $\Ext$ vector spaces between indecomposable injective and projective modules, with multiplicities given by the dimensions of some relative paths that we introduce in Section \ref{cohomology}.

The proofs rely on the reverse procedure, that is on adding arrows. This way we obtain a tensor algebra over the original algebra for a suitable projective bimodule. This algebra is finite dimensional if and only if adding arrows does not create relative cycles. In case of adding just one arrow, which does not leads to a relative cycle, the tensor algebra is a trivial extension which has been already considered in  \cite{GREENPSAROUDAKISSOLBERG}. In this paper E.L. Green, C. Psaroudakis and {\O}. Solberg study this operation in relation to the finitistic dimension conjecture.

In Section \ref{examples} we apply our results to the amalgamation of a bound quiver algebra with a quiver without oriented cycles. In this section we also consider Hochschild cohomology of Gorenstein algebras, by adding arrows to its  quiver.

The results that we obtain in this paper are in the framework of replacing an algebra by a closely related  algebra, towards producing algorithmic ways for computing Hochschild (co)homology. Other work in this direction can be found for instance in \cite{BELIGIANNISJPAA, BUCHW, CIBILS2000, CMRS2002, GREENMARCOSSNASHALL, GREENSOLBERG, HAPPEL, HERMANN, KOENIGNAGASE, MICHELENAPLATZEC}. Hochschild (co)homology has been also considered for bound quiver algebras in relation to their representations and their deformations,  see for instance  \cite{HAPPEL1990,HAPPEL1998,HERMANN,SNASHALLSOLBERG}.

Hochschild (co)homology is a derived invariant, see for instance \cite{RICKARD}. M. Gerstenhaber has shown  in \cite{GERSTENHABER} that Hochschild (co)homology affords additional structure: cup and cap products, and the Gerstenhaber bracket. Together with Connes' differential this constitutes the differential calculus, or Tamarkin-Tsygan calculus of an algebra, see for instance \cite{TAMARKIN TSYGAN,CUNTZSKANDALISTSYGAN}. This theory is also a derived invariant, see \cite{ARMENTAKELLER1,ARMENTAKELLER2}. It should be interesting to describe the effect of deleting  an inert arrow of a bound quiver algebra on its differential calculus.

\normalsize

\section{\sf Relative Hochschild (co)homology and tensor algebras}\label{account}\label{tensor}
Let $k$ be a field and let $B\subset A$ be an extension of $k$-algebras. Modules are left modules.

An $A$-module $P$ is \emph{relative projective} if for any $A$-morphism $f:X\to Y$ with a $B$-section, and any $A$-morphism $g:P\to Y$, there exists an $A$-morphism $g': P\to X$ such that $fg'=g$. The category of $A$-modules is an exact category with respect to the short exact sequences which are $B$-split, see \cite{BUHLER}. As mentioned in \cite{BUHLER}, these notions are commonly attributed to D. Quillen \cite{QUILLEN}. The projective objects of this exact category (see \cite[Definition 11.1]{BUHLER}) are the relative projective modules.

An $A$-module is \emph{induced} if it is isomorphic to $A\otimes_B U$, where $U$ is a $B$-module. As proved by Hochschild in \cite{HOCHSCHILD1956}, a module is relative projective if and only if it is isomorphic to a direct summand of an induced module.

 \begin{defi}\cite{HOCHSCHILD1956}\label{relprojres}
A \emph{relative projective resolution} $P_\bullet \to M$ of an $A$-module $M$ is a complex
  $$\dots\stackrel{d}{\to}P_2\stackrel{d}{\to}P_1\stackrel{d}{\to}P_0\stackrel{d}{\to}M\to 0$$
where $P_i$ is relative projective for all $i$ and there exists a $B$-contracting homotopy, that is there exist $B$-maps $t$ of degree $-1$ such that $dt+td=1$.
  \end{defi}

It can be easily shown that relative projective resolutions are the projective resolutions in the exact category with respect to $B$-split short exact sequences, see \cite[Definition 12.1]{BUHLER}. The usual comparison theorem holds for relative projective resolutions, see \cite[p. 250]{HOCHSCHILD1956} and \cite[Theorem 12.4]{BUHLER}. Hence the following definitions do not depend, up to isomorphism, on the choice of a relative projective resolution.

 \begin{defi}
 Let $M$ and $N$ be $A$-modules and let $P_\bullet \to M$ be a relative projective resolution. The cohomology of the cochain complex
 $$0\to\Hom_A(P_0,N)\to \Hom_A(P_1,N)\to\cdots$$
 is  $\Ext^*_{A|B}(M,N)$.

  Let $M$ be a left $A$-module and let $N$ be a right $A$-module. Let $P_\bullet \to M$ be a relative projective resolution. The homology of the chain complex
  $$\cdots\to N\otimes_A P_1 \to N\otimes_A P_0 \to 0$$
  is $\Tor_*^{A|B}(N,M)$.
  \end{defi}

 Next we recall the relative theory provided by Hochschild in \cite{HOCHSCHILD1956} for bimodules.

  Let $C$ and $D$ be $k$-algebras. The category of $C\!-\!D$-bimodules is identified with the category of  $C\otimes D^{{\mathsf op}}$-modules. A projective $C-D$-bimodule is a projective $C\otimes D^{\mathsf op}$-module.

  \begin{defi}
Let $M$ be an $A$-bimodule. The relative Hochschild cohomology and homology of $A$ respect to $B$ with coefficients in $M$ are respectively the graded vector spaces
  $$\HH^*(A|B, M)=\Ext^*_{(A\!-\!A)|(B\!-\!B)}(A,M)$$
  $$\HH_*(A|B, M)=\Tor_*^{(A\!-\!A)| (B\!-\!B)}(M,A).$$
  \end{defi}

 There is a relative bar resolution of $A$ which provides the subsequent way of computing relative Hochschild (co)homology, see  \cite{HOCHSCHILD1956}. As mentioned in the Introduction, Kaygun obtained long exact sequences relating the usual Hochschild (co)homology with the relative one, see \cite{KAYGUN,KAYGUNe}.

Let $B$ be a $k$-algebra and let $N$ be a $B$-bimodule. The \emph{tensor algebra }
$$T=T_B(N)= B\ \oplus\ N\ \oplus\  N\otimes_B N\ \oplus \cdots \oplus N^{\otimes_B n} \oplus \cdots$$
 is a non negatively graded algebra with $T_0=B$ and $T_n=N^{\otimes_B n}$ for $n>0$. Its \emph{positive part} is $T^{>0}=\oplus_{n>0}N^{\otimes_B n}$.
\begin{exam}\label{kQ}
Let $Q$ be a finite quiver, with set of vertices $Q_0$, set of arrows $Q_1$, and $s,t:Q_1\to Q_0$ the maps which associate to each arrow respectively its source and target vertices. The path algebra $kQ$ is the tensor algebra over the semisimple commutative algebra $kQ_0$ of the $kQ_0$-bimodule $kQ_1$. We denote by $Q_n$ the set of paths of length $n$, note that $(kQ)_n =kQ_n$.
\end{exam}

A \emph{bound quiver algebra} is an algebra $kQ/ I$, where $I$ is an \emph{admissible} ideal, that is  $\langle Q_n\rangle \subset I \subset \langle Q_2\rangle$ for some $n$.

Let $k$ be an algebraically closed field.  By results of P. Gabriel in \cite{GABRIEL1973,GABRIEL1980}, see also for instance  \cite[Theorem 3.7]{ASSEMSIMSONSKOWRONSKY} or \cite{SCHIFFLER}, a finite dimensional $k$-algebra $A$ is Morita equivalent to a bound quiver algebra.

We recall the universal property of a tensor algebra $T=T_B(N)$. Let $\Lambda$ be a $k$-algebra. An algebra morphism $\varphi :T\to \Lambda$ is uniquely determined by an algebra morphism $\varphi_0: B\to \Lambda$ - which turns $\Lambda$ into a $B$-bimodule - and a $B$-bimodule morphism $\varphi_1: N\to \Lambda$.

\begin{theo}\label{cero}
Let $T=T_B(N)$ be the tensor algebra over a $k$-algebra $B$ of a $B$-bimodule $N$.  Let $X$ be a $T$-bimodule. For $*\geq 2$,
$$\HH^*(T|B,X)= 0,$$$$\HH_*(T|B,X)=0.$$
\end{theo}
\begin{proof}
We will show that the following is a $(T\!-\!T|B\!-B)$ relative projective resolution of $T$:
\begin{equation}\label{Therelprojres}
0\to T\otimes_B N\otimes_BT\stackrel{g}{\to} T\otimes_B T \stackrel{f}{\to} T\to 0
\end{equation}
where $f$ is the product of $T$ and $g(x\otimes n\otimes y)=xn\otimes y - x\otimes ny.$ We obtain the results of the statement by using this relative resolution to compute Hochschild (co)homology.

It is straightforward to verify that $fg=0$. Clearly the involved $T$-bimodules are induced, hence they are relative projective.

In order to define a $(B\!-\!B)$-contracting  homotopy, we introduce notation. Firstly, instead of $x_1\otimes\dots \otimes x_n$ we will write $x_1,\dots, x_n$. Secondly observe that the tensor element $x_1,\dots, x_n$ can be considered in different ways: for any decomposition $n=i+j$ with $i\geq 0$ and $j\geq 0$, it can be viewed as an element of
$$N^{\otimes_B i}\otimes_B N^{\otimes_B j} \subset \left(T\otimes_B T\right)_n = \bigoplus_{k+l=n}\left(N^{\otimes_B k}\otimes_B N^{\otimes_B l}\right) \subset T\otimes_B T.$$
In this case we write this element as
$(x_1,{\dots},x_i\stackrel{\triangledown}{,}x_{i+1},\dots,x_n).$
For $n=0+n$ and $n=n+0$ we remark that the notation is:
$$(\stackrel{\triangledown}{,}\!x_1,\dots,x_n)\in B\otimes_B N^{\otimes_Bn}\subset (T\otimes_B T)_n,$$
$$(x_1,\dots,x_n\!\!\stackrel{\triangledown}{,})\in  N^{\otimes_Bn}\otimes_BB\subset (T\otimes_B T)_n.$$
Finally for $n=i+1+j$ with $i\geq 0$ and $j\geq 0$, the tensor $x_1,\dots,x_n$ can also be considered as an element of
$$
\arraycolsep=1mm\def\arraystretch{2}
\begin{array}{lll}
N^{\otimes_Bi}\otimes_BN\otimes_BN^{\otimes_Bj}&\subset (T\otimes_BN\otimes_BT)_n\\&=\bigoplus_{k+1+l=n}\left(N^{\otimes_B k}\otimes_BN\otimes_B N^{\otimes_Bl}\right)\\
&\subset T\otimes_BN\otimes_BT.
\end{array}
$$
In this case we denote it by $x_1,\dots,x_i,\stackrel{\triangledown}{x_{i+1}},x_{i+2},\dots, x_n$. For $n=0+1+(n-1)$ and $n=(n-1)+1+0$ we write:
$$\stackrel{\triangledown}{x_1},\dots,x_n \in B\otimes_B N\otimes_B N^{\otimes_B (n-1)}\  \subset \ (T\otimes_B N\otimes_B T)_n,$$
$$x_1,\dots,\stackrel{\triangledown}{x_n} \in N^{\otimes_B (n-1)}\otimes_B N\otimes_B B\  \subset \ (T\otimes_B N\otimes_B T)_n.$$
With these notations the morphisms $f$ and $g$ are as follows:
$$f(x_1,\dots,x_i,\stackrel{\triangledown}{,}x_{i+1},\dots,x_n)=x_1,\dots,x_n$$
$$
\begin{array}{lll}
g(x_1,\dots,x_i,\stackrel{\triangledown}{x_{i+1}},x_{i+2},\dots, x_n)= (x_1,\dots,x_{i+1}\stackrel{\triangledown}{,}\dots,x_n)-(x_1,{\dots}\stackrel{\triangledown}{,}x_{i+1},\dots,x_n).
\end{array}$$
The sequence (\ref{Therelprojres}) is graded, it is the direct sum of the following sequences of $B$-bimodules:
\begin{equation*}
0\longrightarrow (T\otimes_B N\otimes_BT)_n\stackrel{g}{\longrightarrow} (T\otimes_B T)_n \stackrel{f}{\longrightarrow} T_n\longrightarrow 0
\end{equation*}
Next we define $s:T_n\to (T\otimes_B T)_n$ and $r:(T\otimes_B T)_n\to (T\otimes_B N\otimes_BT)_n$ and we  prove that these maps verify the following:
\begin{itemize}
\item $r$ and $s$ are  $B\!-\!B$-bimodule maps,
\item $fs=1$ and $rg=1$,
\item $gr+sf = 1$
\end{itemize}
providing this way the required contracting homotopy, see Definition \ref{relprojres}.

The morphism $s$ is defined by  $s(x_1,\dots, x_n)= (\stackrel{\triangledown}{,}\!x_1,\dots,x_n)$, note that it is well defined. Moreover $s$ is easily seen to be a $B\!-\!B$-section of $f$.

In order to define $r$, let $i\geq 0$, $j\geq 0$ with $i+j=n$. Let $$r_i:  N^{\otimes_B i}\otimes_B N^{\otimes_B j} \to (T\otimes_B N\otimes_B T)_n$$ be given by
$$r_i(x_1,\dots,x_i,\stackrel{\triangledown}{,}x_{i+1},\dots,x_n)=\sum_{k=1}^i (x_1,{\dots},\stackrel{\triangledown}{x_{k}},\dots, x_n).$$
Observe that $r_0=0$. Each $r_i$ is a well defined $B\!-\!B$-morphism. We put $r=\oplus_{i=0}^nr_i$. Next we verify that $rg=1$.
$$
\arraycolsep=1mm\def\arraystretch{2}
\begin{array}{llll}
rg(x_1,\dots,\stackrel{\triangledown}{x_{i+1}},\dots, x_n)=\\
r_{i+1}(x_1,\dots, x_{i+1}\stackrel{\triangledown}{,}\dots,x_n)-r_i(x_1,{\dots} \stackrel{\triangledown}{,}x_{i+1},\dots,x_n)=\\
\sum_{k=1}^{i+1} (x_1,{\dots}\stackrel{\triangledown}{x_{k}},\dots, x_n)- \sum_{k=1}^i (x_1,{\dots}\stackrel{\triangledown}{x_{k}},\dots, x_n)=\\
(x_1,\dots,\stackrel{\triangledown}{x_{i+1}},\dots, x_n).
\end{array}
$$

Finally we verify that $gr+sf=1$.
$$sf(x_1,\dots,x_i\stackrel{\triangledown}{,}x_{i+1},\dots,x_n)=s(x_1,\dots, x_n)= (\stackrel{\triangledown}{,}\!x_1,\dots,x_n).$$
$$
\arraycolsep=1mm\def\arraystretch{2}
\begin{array}{llll}
gr(x_1,\dots,x_i\stackrel{\triangledown}{,}x_{i+1},\dots,x_n)=g\left(\sum_{k=1}^i (x_1,\dots,\stackrel{\triangledown}{x_{k}},\dots, x_n)\right)=\\
\sum_{k=1}^i\left( (x_1,\dots,x_{k}\stackrel{\triangledown}{,}\dots,x_n)-(x_1,{\dots}\stackrel{\triangledown}{,}x_{k},\dots,x_n)\right)=\\
-(\stackrel{\triangledown}{,}\!x_1,\dots,x_n)+(x_1,\dots,x_i\stackrel{\triangledown}{,}x_{i+1},\dots,x_n).
\end{array}
$$
In particular $(gr+sf)(\stackrel{\triangledown}{,}\!x_1,\dots,x_n)=0+(\stackrel{\triangledown}{,}\!x_1,\dots,x_n).$
\qed
\end{proof}

\section{\sf Homology}\label{homology}

Let $A=kQ/I$ be a bound quiver algebra, and let $R$ be a minimal finite subset of the path algebra $kQ$ such that $I=\langle R\rangle$. Note that elements $r\in I$ are possibly linear combination of parallel paths, that is of paths with the same source and the same target vertices, see for instance \cite{ASSEMSIMSONSKOWRONSKY}.
 \begin{defi}
 Let $A=kQ/\langle R \rangle$ be a bound quiver algebra. An arrow $a$ of $Q$ is \emph{inert} if $a$ does not appear in any of the paths which provide the elements of $R$.
\end{defi}

Let $D$ be a set of inert arrows, let $Q_{{\setminus}D}$ be the quiver $Q$ where the arrows of $D$ are deleted. We set
\begin{equation}\label{deleted}
A_{\setminus D}=k(Q_{{\setminus}D})/\langle R\rangle_{k(Q_{{\setminus}D})}
\end{equation}
where the denominator is the two sided ideal generated by $R$ in $k(Q_{{\setminus}D})$. Note that $A_{{\setminus}D}$ is a bound quiver algebra. Moreover $A_{{\setminus}D}$ is a subalgebra of $A$.

 \begin{theo}
Let  $A=kQ/\langle R \rangle$ be a bound quiver algebra and let $D$ be a set of inert arrows of $Q$. For $*\geq 2$

$$\HHH_*(A_{\setminus D})\simeq \HHH_*(A).$$

 \end{theo}

The proof will rely on the reverse procedure, namely adding new arrows.

\begin{defi}
Let $Q$ be a quiver and let $F$ be a finite set of \emph{new arrows}, that is $F$ is a finite set with maps $s,t: F\to Q_0$. The \emph{new quiver} $Q_F$ is given by $(Q_F)_0=Q_0$ and $(Q_F)_1 =Q_1 \cup F$, where the source and target maps are provided by the corresponding maps of $Q_1$ and $F$.

Let $B=kQ/I$ be a bound quiver algebra, and let $F$ be a set of new arrows. We set
$$B_F= kQ_F/\langle I\rangle_{kQ_F}.$$
\end{defi}

\begin{rema}
Suppose that $B_F$ is finite dimensional. Obviously the set of arrows $F$ is inert and $(B_F)_{\setminus F} =B$.

If $A=kQ/\langle R\rangle$ is a bound quiver algebra and $D$ is a set of inert arrows, then $\left(A_{\setminus D}\right)_D=A$.
\end{rema}

Given a bound quiver algebra $B$ and a set of new arrows $F$, consider the projective $B$-bimodule
\begin{equation}\label{N}
N= \bigoplus_{a\in F} Bt(a)\otimes s(a)B.
\end{equation}

The following result is clear by using the universal property of tensor algebras.
\begin{theo}
Let $B=kQ/\langle R \rangle$ be a bound quiver algebra, and let $F$ be a set of new arrows. The algebra $B_F$ is isomorphic to the tensor algebra $T=T_B(N)$.
\end{theo}

\begin{theo}Let $B$ be a bound quiver algebra, and let $F$ be a set of new arrows. For $*\geq 2$ $$\HHH_*(B_F)\simeq \HHH_*(B).$$\end{theo}

\begin{proof}
By the previous result,  $B_F\simeq T=T_B(N)$. The $B$-bimodule $T/B = T^{>0}$
is projective, hence the Jacobi-Zariski long exact sequence for Hochschild homology obtained by Kaygun holds, see \cite{KAYGUN,KAYGUNe}.
On the other hand, by Theorem \ref{cero} we have that $\HH_*(T|B, T)=0$ for $*\geq 2$. Then
$\HHH_*(T)\simeq \HH_*(B,T)$ for $*\geq 2$. Moreover $\HH_*(B,T)=\HHH_*(B)\oplus \HH_*(B, T^{>0})$. The second summand is zero since the $B$-bimodule of coefficients is projective.\qed
\end{proof}

\section{\bf \sf Cohomology}\label{cohomology}

 In case of deleting one inert arrow we will obtain the following result. Note that for $X$ a right module, we denote by $X'$ the left module $\Hom_k(X,k)$.

\begin{theo}\label{deletone}
Let $A$ be a bound quiver algebra, let $a$ be an inert arrow from $e$ to $f$ and let $B=A_{\setminus\{a\}}$. For $*\geq 2$ we have
$$\dim_k \HHH^*(B)= \dim_k \HHH^*(A) -  \dim_k \Ext_A^*((fA)', Ae).$$
\end{theo}
See Remark \ref{proofonearrow} for a proof of the previous result.

Next we state a result for a set of deleted arrows. As in the previous section, the proof will be by the reverse procedure of adding arrows.  However the Jacobi-Zariski long exact sequence for cohomology requires that the bimodule of coefficients is finite dimensional. To ensure this, and to state the result, we introduce the following.

Let $A=kQ/\langle R\rangle$ be a bound quiver algebra, let $D$ be a set of inert arrows and let $B=A_{{\setminus}D}$, see (\ref{deleted}).
\begin{itemize}
\item[-] A pair of arrows $(a_2, a_1)$ of $D$ is \emph{linked}  if $s(a_2)Bt(a_1) \neq 0$.

\item[-] A \emph{relative $n$-path} is a sequence  $\gamma=a_na_{n-1}\dots a_2a_1$ of arrows of $D$
such that $(a_{i+1},a_i)$ is linked for $i=1,\dots, n-1$. The \emph{source} $s(\gamma)$ and the \emph{target} $t(\gamma)$ of $\gamma$ are respectively $s(a_1)$ and $t(a_n)$.

\item[-] The set of relative $n$-paths is denoted $\Gamma_n$, while the set of all relative paths is denoted $\Gamma$.

\item[-] The relative path $\gamma$ is a \emph{relative cycle} if $(a_1,a_n)$ is linked. This way an arrow $a$ is a relative cycle if $(a,a)$ is linked, that is if $s(a)Bt(a)\neq 0$.

\item[-] We set
$$\dim_k\gamma = \prod_{i=1}^{n-1}\dim_k s(a_{i+1})Bt(a_{i}).$$
If $\gamma = a\in F$, then $\dim_k\gamma =1$.
\end{itemize}

We will prove the following as a consequence of Theorem \ref{cohonew}.

\begin{theo}\label{deletefamily}
Let $A$ be a bound quiver algebra, let $D$ be a set of inert arrows and let $B=A_{\setminus D}$. For $*\geq 2$ we have
$$\dim_k \HHH^*(B)= \dim_k \HHH^*(A) - \sum_{\gamma\in\Gamma}\dim_k\gamma \ \dim_k \Ext_B^*((t(\gamma)B)', Bs(\gamma)).$$
\end{theo}

\begin{rema}\label{proofonearrow}
The proof of Theorem \ref{deletone} follows: if $D=\{a\}$, then there is only one relative path $\gamma = a$. The formula above gives
$$\dim_k \HHH^*(B)= \dim_k \HHH^*(A) - \dim_k\Ext_B^*((t(a)B)', Bs(a)).$$
Moreover, at the end of this section we will prove that $$\Ext_A^*((t(a)A)', As(a))\simeq \Ext_B^*((t(a)B)', Bs(a)).$$
\end{rema}

We first prove the following.
\begin{prop}\label{no relative cycles}
Let $B$ be a bound quiver algebra, let $F$ be a set of new arrows and let $N$ be the associated projective $B$-bimodule, see (\ref{N}). We have
$$N^{\otimes_Bn} \simeq\bigoplus_{\gamma\in\Gamma_n}\dim_k\gamma \ Bt(\gamma)\otimes s(\gamma)B.$$
The algebra $T_B(N)=B_F$ is finite dimensional if and only if the length of the relative paths is bounded, that is there are no relative cycles.
\end{prop}
\begin{proof}
We sketch the proof in case $F=\{a_1,a_2\}$, where $a_1$ is from $e$ to $f$ and $a_2$ from $g$ to $h$, hence $N=(Bf\otimes eB) \oplus (Bh\otimes gB)$.
One of the four direct summands of $N\otimes_B N$ is
$$(Bh\otimes gB)\otimes_B (Bf\otimes eB)= Bh\otimes gBf\otimes eB$$
which is isomorphic as a $B$-bimodule to $\dim_k(gBf) Bh\otimes eB.$ If $(a_2,a_1)$ is linked, then this direct summand is non zero and it corresponds  to the  relative path $a_2a_1$ in the formula. The rest of the proof is along the same lines.

\end{proof}

\begin{lemm}\label{ce}
Let $B$ be a bound quiver algebra, and let $e,f\in Q_0$. We have
$$\HH^*(B, Bf\otimes eB)=\Ext^*_B((eB)', Bf)$$
\end{lemm}
\begin{proof}
First we recall that if $U$ and $W$ are left $B$-modules, then $$\HH^*(B, \Hom_k(W,U))\simeq\Ext^*_B(W,U),$$
see for instance \cite[p. 170, 4.4]{CARTANEILENBERG}).

Let $U$ and $V$ be respectively left and right finite dimensional $B$-modules, so that the $B$-bimodules
$U\otimes V$ and $\Hom_k(V',U)$ are isomorphic. Then
$$\HH^*(B, U\otimes V)\simeq \Ext^*_B(V',U)$$\qed
\end{proof}
\begin{theo}\label{cohonew}
Let $B= kQ/\langle R\rangle$ be a bound quiver algebra, let $F$ be a set of new arrows and let $A=B_F=T_B(N)$. Suppose that there are no relative cycles, that is $A$ is finite dimensional.  For $*\geq 2$
$$\HHH^*(A)= \HHH^*(B)\oplus \bigoplus_{\gamma\in\Gamma} \dim_k\gamma\  \Ext^*_B(({s(\gamma)}B)', B{t(\gamma)}).$$
\end{theo}
\begin{proof}
We infer from the Jacobi-Zariski long exact sequence obtained by Kaygun in \cite{KAYGUN,KAYGUNe} and from Theorem \ref{cero} that for $*\geq 2$ we have $\HHH^*(A)\simeq \HH^*(B,A).$ Moreover $ \HH^*(B,A) \simeq \HHH^*(B) \oplus \HH^*(B, T^{>0}).$
By Proposition \ref{no relative cycles}
$$\HH^*(B, N^{\otimes_Bn})\simeq \bigoplus_{\gamma \in \Gamma_n}\dim_k\gamma\ \HH^*(B, Bt(\gamma)\otimes s(\gamma)B).$$
Lemma \ref{ce} provides the result.
\qed
\end{proof}

Theorem \ref{deletefamily} is inferred from the previous result by adding as new arrows the deleted ones.

\begin{lemm}\label{flechaunica}
Let $B$ be a bound quiver algebra. Let $a$ be a new arrow from $e$ to $f$ which does not provides a relative oriented cycle, and let $A=B_F=T_B(N)$. For $*\geq 2$
$$\Ext_A^*((fA)', Ae)\simeq \Ext_B^*((fB)', Be).$$
\end{lemm}
\begin{proof}
The Kaygun's Jacobi-Zariski long exact sequence and Theorem \ref{cero}, show that for $*\geq 2$
$$\HHH^*(A, Af\otimes eA)\simeq \HHH^*(B, Af\otimes eA).$$
We have that $N\otimes_BN=0$ because $a$ is not a relative oriented cycle, that is $eBf=0$. Hence $A= B\oplus N$, where $N=Bf\otimes eB$. A simple computation shows that $Af\otimes eA= Bf\otimes eB$. Hence
$$\HHH^*(B, Af\otimes eA) = \HHH^*(B, Bf\otimes eB).$$
Lemma \ref{ce} provides the result.
\qed
\end{proof}

We record the following result as a consequence of Theorem \ref{deletone}.
\begin{coro}
Let $A$ be a bound quiver algebra, let $a$ be an inert arrow from $e$ to $f$ and let $B=A_{\setminus\{a\}}$. If $Ae$ is an injective module, and/or if $(fA)'$ is a projective module, we have for $*\geq 2$
$$\HHH^*(B)\simeq \HHH^*(A).$$
\end{coro}

In order to use the previous result, we recall the well known criterion for a projective indecomposable module to be injective.

\begin{prop}
Let $A=kQ/I$ be a bound quiver algebra and let $e\in Q_0$.
The indecomposable projective left module $Ae$ is injective  if and only if $\mathsf{soc}Ae$ is a simple module $S_y$ for a vertex $y$, and $\dim_k (Ae) = \dim_k (yA)'$. In that case $Ae\simeq (yA)'$.
\end{prop}

\section{\sf Examples}\label{examples}

\subsection{\sf Amalgamation with a quiver without oriented cycles}

Let $\Lambda=kQ/I$ be a bound quiver algebra, and let $S$ be a finite quiver without oriented cycles. We will  construct a quiver where some vertices of $S$ and $Q$ are identified.

Let $Y\subset S_0$ and  $X\subset Q_0$, with  a bijective map $u:X\to Y$. Let $\sim$ be the equivalence relation on $X\cup Y$ given by $x\sim u(x)$. We extend it to an equivalence relation on $Q_0\cup S_0$ in the obvious manner, each other vertex is just equivalent to itself.

\begin{defi}
The \emph{amalgamated quiver} $Q\cup_u S$ is given by
\begin{itemize}
\item $Q_0\cup_u S_0/\sim$,
\item $(Q\cup_u S)_1= Q_1\cup S_1$.

The \emph{amalgamated algebra} $\Lambda\cup_uS$ is $k(Q\cup_u S)/\langle I\rangle_{k(Q\cup_u S)}$.
\end{itemize}
\end{defi}

\begin{theo}
Let $\Lambda=kQ/I$ be a bound quiver algebra, let $S$ be a finite quiver without oriented cycles and let $u$ be as above. Suppose that for each path of positive length of $S$, one of its extremities is not in $Y$.
For $*\geq 2$
$$\HHH^*(\Lambda\cup_u S)\simeq\HHH^*(\Lambda).$$
\end{theo}

\begin{proof}
Let  $G$ be the quiver $Q$ with the set $S_0\setminus Y$ of new vertices added - each one is a connected component of $G$. Consider the algebra $B=\Lambda\times \mbox{{\Large$\times$}}_{|S_0\setminus Y|}k$. Hochschild cohomology is additive on the product of algebras, and  $\HHH^*(k)= 0$ for $*\geq 1$, hence $\HHH^*(\Lambda) =\HHH^*(B)$ for $*\geq 1$.

Recall that  $G_{S_1}$ is the quiver $G$ with the set $S_1$ of new arrows added, we have $G_{S_1}=Q\cup_u S$. The hypotheses on the paths of $S$ imply that there are no relative oriented cycles, hence $B_{S_1}=\Lambda\cup_uS$ is finite dimensional. Moreover, for each relative path at least one of its extremities is an isolated vertex $x\in S_0\setminus Y$. We have that $(xB)'=Bx$, which is a simple module.  The corresponding $\Ext_B$'s in  Theorem \ref{cohonew} vanish. \qed
\end{proof}

\subsection{\sf Gorenstein algebras}

A finite dimensional $k$-algebra $B$ is \emph{Gorenstein} if $B$  is of finite injective dimension $n_1$ as a left module, and the injective left module $B'$ is of finite projective dimension $n_2$, see for instance \cite{HAPPEL1991}. The algebra is called $n$-Gorenstein for $n$ the maximum of these numbers.

Let $B=kQ/\langle R\rangle$ be a bound quiver algebra. If $B$ is $n$-Gorenstein, then for any $x\in Q_0$ the projective dimension of $(xB)'$ and the injective dimension of $Bx$ is at most $n$.   For instance selfinjective algebras are $0$-Gorenstein algebras. The following result is  a consequence of  Theorem \ref{cohonew}.

\begin{prop}
Let $B=kQ/I$ be a bound quiver algebra which is $n$-Gorenstein. Let $\widehat Q$ be the quiver obtained from $Q$ by adding a finite number of new arrows, and suppose $\widehat B= k\widehat  Q/\langle I\rangle_{\widehat Q}$ is finite dimensional. For $n>0$ and $*\geq n+1$,
$$\HHH^*(\widehat  B)\simeq\HHH^*(B)$$
while for $n=0$ - that is if $B$ is selfinjective- the isomorphism holds for $*\geq 2$.
\end{prop}

\begin{exam}
For $m\geq 2$, let $Q$ be the following cyclic connected quiver:

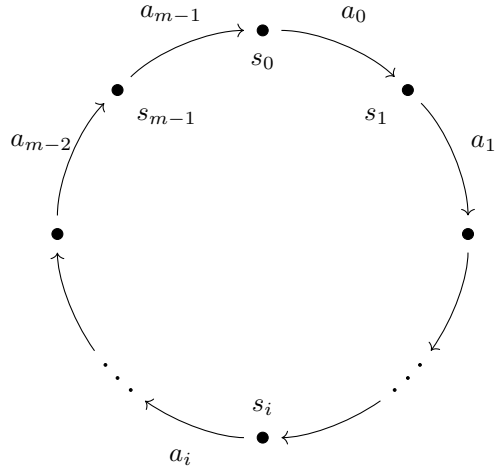
\begin{center}
\begin{tikzpicture}
[scale=.9]

\foreach \ang\lab\anch in {90/s_0/north, 45/s_1/{north east}, 0/ /east, 270/s_i/south, 180/ /west, 135/s_{m-1}/{north west}}{
  \draw[fill=black] ($(0,0)+(\ang:3)$) circle (.08);
  \node[anchor=\anch] at ($(0,0)+(\ang:2.8)$) {$\lab$};
}

\foreach \ang\lab in {90/0,45/1,180/{m-2},135/{m-1}}{
  \draw[->,shorten <=7pt, shorten >=7pt] ($(0,0)+(\ang:3)$) arc (\ang:\ang-45:3);
  \node at ($(0,0)+(\ang-22.5:3.5)$) {$a_{\lab}$};
}

\draw[->,shorten <=7pt] ($(0,0)+(0:3)$) arc (360:325:3);
\draw[->,shorten >=7pt] ($(0,0)+(305:3)$) arc (305:270:3);
\draw[->,shorten <=7pt] ($(0,0)+(270:3)$) arc (270:235:3);
\draw[->,shorten >=7pt] ($(0,0)+(215:3)$) arc (215:180:3);
\node at ($(0,0)+(0-20:3.5)$) { };
\node at ($(0,0)+(315-25:3.5)$) { };
\node at ($(0,0)+(270-20:3.5)$) {$a_i$};
\node at ($(0,0)+(225-25:3.5)$) { };

\foreach \ang in {310,315,320,220,225,230}{
  \draw[fill=black] ($(0,0)+(\ang:3)$) circle (.02);
}

\end{tikzpicture}
\end{center}

For $2\leq l<m$, let $B=kQ/\langle Q_l\rangle$, which is a selfinjective algebra.  Let $\widehat Q$ be the quiver obtained from $Q$ by adding a finite number of arrows at pairs of vertices. Suppose that $\widehat B= k\widehat Q/\langle Q_l\rangle_{\widehat Q}$ is finite dimensional, that is there are no relative oriented cycles. For $*\geq 2$ $$\HHH^*(\widehat B)\simeq\HHH^*(B).$$

\end{exam}

\footnotesize
\noindent C.C.:\\
Institut Montpelli\'{e}rain Alexander Grothendieck, CNRS, Univ. Montpellier, France.\\
{\tt Claude.Cibils@umontpellier.fr}

\medskip
\noindent M.L.:\\
Instituto de Matem\'atica y Estad\'\i stica  ``Rafael Laguardia'', Facultad de Ingenier\'\i a, Universidad de la Rep\'ublica, Uruguay.\\
{\tt marclan@fing.edu.uy}

\medskip
\noindent E.N.M.:\\
Departamento de Matem\'atica, IME-USP, Universidade de S\~ao Paulo, Brazil.\\
{\tt enmarcos@ime.usp.br}

\medskip
\noindent A.S.:
\\IMAS-CONICET y Departamento de Matem\'atica,
 Facultad de Ciencias Exactas y Naturales,\\
 Universidad de Buenos Aires, Argentina. \\{\tt asolotar@dm.uba.ar}

\end{document}